\documentclass[11pt]{article}
\usepackage[latin1]{inputenc}
\usepackage{amsmath}
\usepackage{amsfonts}
\usepackage{amssymb}
\usepackage{wasysym}
\usepackage{amsthm}

\newcommand{\R}{\mathbb R}
\newcommand{\eps}{\epsilon}
\newcommand{\oX}{\overrightarrow{X}}

\newcommand{\vp}{\vec{p}}
\newcommand{\vq}{\vec{q}}
\newcommand{\vm}{\vec{m}}
\newcommand{\In}{\{1\ldots N\}}
\newcommand{\vz}{\vec{\zeta}}

\newcommand{\xiphi}{\xi_\phi}
\newcommand{\vo}{\vec{0}}
\newcommand{\vwP}{{\cal P}^{\psi,w}_m}

\newcommand{\uvP}{\underline{\cal P}^\psi_m}
\newcommand{\vphi}{\vec{\phi}}
\newcommand{\vmu}{\vec{\mu}}
\newcommand{\uX}{\underline{X}}
\newcommand{\vpsi}{\vec{\psi}}

\newcommand{\vmS}{S^\psi}
\newcommand{\uvmS}{\underline{S}^\psi}

\newcommand{\vmSw}{S^{ \vec{\psi,w}}}
\newcommand{\uvmSw}{\underline{S}^{ \vec{\psi,w}}}
\newcommand{\uXi}{{\Xi_\phi^+}}
\newtheorem{theorem}{Theorem}[section]
\newtheorem{lemma}{Lemma}[section]
\newtheorem{proposition}{Proposition}[section]
\newtheorem{cor}{Corollary}[section]

\newtheorem{defi}{Definition}[section]
\newtheorem{acknowledgment*}{Acknowledgment}
\newtheorem{assumption}{Assumption}[section]
\newcommand{\be}{\begin{equation}}
\newcommand{\ee}{\end{equation}}

\begin{document}

\setcounter{page}{1}
\setlength{\baselineskip}{1.3\baselineskip}
\thispagestyle{empty}





\huge
\begin{center}{\bf On Semi-discrete Monge Kantorovich and generalized partitions}\end{center}
\normalsize
\begin{center} Gershon Wolansky\end{center}
\begin{abstract}
Let  $X$ a probability measure space and $\psi_1....\psi_N$ measurable, real  valued
 functions on $X$. Consider all  possible partitions of
$X$ into $N$ disjoint subdomains $X_i$ on which $\int_{X_i}\psi_i$ are prescribed. We address  the question of characterizing the set $(m_1,,,m_N) \in \R^N$ for  which there exists a partition $X_1, \ldots X_N$ of $X$ satisfying  $\int_{X_i}\psi_i= m_i$  and discuss some optimization problems on this  set of  partitions. The relation of this problem to semi-discrete version of optimal mass transportation is discussed as well.
\end{abstract}
\tableofcontents
\section{Introduction}\label{int}
\subsection{Semi-discrete Monge problem}
 Optimal Transportation, known also as {\it Monge-Kantorovich theory}, became very popular in last decades. The first publication by Monge [\ref{mon}] goes back to 1781. Excellent modern reviews are the books of C. Villani [\ref{vil}, \ref{vil1}].

The object of optimal transportation  is  to find an  optimal map transporting a given, prescribed  probability measure into another. In general setting, it deals with a pair of probability measure spaces $(X,{\cal B}_X, \mu), (Y,{\cal B}_Y, \nu)$ and a $\mu\oplus\nu$ measurable cost function $c:X\times Y\rightarrow \R$. The Monge problem  is to maximize\footnote{Traditionally, the MK problem deals with minimization of the cost. In the current setting it is more natural to talk about maximization. The two options are, of course, equivalent under a sign change of the cost $c$ \ . } the  functional
\be\label{mk} T\rightarrow  \int_Xc(x, T(x))d\mu\in \R\ee
among all measurable maps $T:(X, {\cal B}_X)\rightarrow (Y, {\cal B}_Y)$ which  transport the measure $\mu$ to $\nu$, i.e. $T_\#\mu=\nu$, that is
\be\label{mkT}\mu(T^{-1}(B))=\nu(B)\ee
for any  $B\in {\cal B}_Y$.

In the special case where  $Y$ is a {\it finite} space, i.e $Y:=\{ y_1, \ldots y_N\}$ and ${\cal B}_Y=2^{Y}$, $\nu$ is characterized by a vector
\be\label{simI} \vm:=(m_1, \ldots, m_N)\in S_I:= \left\{ \vm\in \R^{N} \ , \ \ \sum_{i\in I} m_i=1 \ \ , \ \ m_i\geq 0 \ \right\}  \ee
via   $\nu(\{y_i\}):= m_i$. Here  and thereafter, $I:= \{1\ldots N\}$.

 In this case, any mapping $T_\#\mu=\nu$  induces a {\it partition} of $X$ into a finite number of components
$X_i:= T^{-1}(\{y_i\})\in {\cal B}_X$ where $\mu(X_i)=m_i$. The optimal transport plan $T$ is then reduced to an {\it optimal partition}\footnote{ See [\ref{Ri}]. } of $X$ within the class
 \begin{multline}\label{partstrong}{\cal P}_{\vm}:= \{\overrightarrow{X}:=(X_1, \ldots X_N) \ ; \ \  X_i\in {\cal B}_X \ \\
  \cup_1^N X_i=X \ , \ \  \ \mu(X_i\cap X_j)=0 \ \text{if} \ i\not= j, \ \ \mu(X_i)=m_i \} \ .  \end{multline}

In the above case we can replace $c:X\times Y\rightarrow \R$ by $N$ measurable functions $\phi_i:X\rightarrow \R$ via $\phi_i(x):= c(x, y_i)$.
The semi-discrete (or  {\it optimal partition})  Monge problem  of maximizing (\ref{mk}, \ref{mkT}) takes the form
\be\label{md}\Xi_\phi^*(\vm):=  \sup_{\overrightarrow{X}}\left\{ \sum_1^N\int_{X_i} \phi_i(x)d\mu\  \ ;  \ \ \  \overrightarrow{X}\in {\cal P}_{\vm} \right\} , \ee
where, again,   $\vm\in S_I$.
\par
This paper  generalizes the concept of optimal partition in three  directions to be described below.
\subsubsection{Individual prices}\label{sip}
Let $\vec{\psi}:= (\psi_1\ldots \psi_N)$ where $\psi_i:X\rightarrow \R$ are  measurable functions on $(X, {\cal B}_X)$.
Let
  \begin{multline}\label{part}{\cal P}^{ \vec{\psi}}_{\vm}:=   \{\overrightarrow{X}:=(X_1, \ldots X_N) \ ; \ \  X_i\in {\cal B}_X \\ \cup_1^N X_i=X \ , \ \   \ \mu(X_i\cap X_j)=0 \ \text{if} \ i\not= j, \ \ \int_{X_i}\psi_i d\mu=m_i \} \end{multline}
 and set
 \be\label{SN} \vmS_I:= \left\{ \vm\in \R^N \ \ ; \ \ \ {\cal P}^{\vec{\psi}}_{\vm}\not=\emptyset\right\} \ . \ee
The {\it generalized} optimal partition problem (\ref{md})  takes the form of
\be\label{gmk}\Xi_\phi^*(\vm):=  \sup_{\overrightarrow{X}}\left\{ \sum_1^N\int_{X_i} \phi_i(x)d\mu \ \ ; \ \ \ \overrightarrow{X}\in {\cal P}_{\vm}^{\vec{\psi}}\right\}\ee
where $\vec{m}\in \vmS_I$.
\subsubsection {Subpartition}\label{ssubp}
The definition of ${\cal P}^{\vpsi}_{\vm}$ requires the partition to exhaust  the space $X=\cup_1^N X_i$.
We extend the set of partitions ${\cal P}^{ \vec{\psi}}_{\vm}$ to {\it sub partitions} where  $ \cup_1^N X_i\subseteq X$:
 \begin{multline}\label{subpart}\underline{\cal P}^{ \vec{\psi}}_{\vm}:=   \{\overrightarrow{X}:=(X_1, \ldots X_N) \ ; \ \  X_i\in {\cal B}_X \\ \cup_1^N X_i\subseteq X \ , \ \   \ \mu(X_i\cap X_j)=0 \ \text{if} \ i\not= j, \ \ \int_{X_i}\psi_i d\mu=m_i \} \end{multline}
 and, respectively,
 \be\label{subSN} \uvmS_{I}:= \left\{ \vm\in \R^N \ \ ; \ \ \ \underline{\cal P}^{\vec{\psi}}_{\vm}\not=\emptyset\right\} \ , \ee
 \be\label{subgmk}\uXi^*(\vm):=  \sup_{\overrightarrow{X}}\left\{ \sum_1^N\int_{X_i} \phi_i(x)d\mu \ \ ; \ \ \ \overrightarrow{X}\in \underline{\cal P}_{\vm}^{\vec{\psi}}\right\}
  \ \ \ , \ \  \vm \in\uvmS_{I}\ee

\subsubsection{ Optimal selections}\label{sselection}
To motivate the above we consider the following cooperative game:\par
Let $\{X,{\cal B},  \mu\}$ be a probability measure space (the "cake").
\par
For each agent $i\in\In$ and $x\in X$  we associate the {\it price} $\psi_i(x)\in\R$ of purchase of   $x$ by the the agent $i$.
\par
Let $C_i\geq 0$ be the {\it capital} of agent $i$, we set  $\overrightarrow{C}=\{ C_1\ldots C_N\}\in \R^N$. An {\it affordable share} for $i$ is a part of the cake
$X_i\in {\cal B}$ such that  $\int_{X_i}\psi_i d\mu\leq C_i$. An {\it  admissible partition} of  $X$  is defined as a partition of $X$ into $N$ essentially disjoint affordable shares  of the agents $\overrightarrow{X}:=(X_1, \ldots X_N)$, that is
$$
\mu(X_i\cap X_i)=0 \ \ \text{if} \ \ i\not= j \ \ \ ; \ \ \
\int_{X_i}\psi_i d\mu \leq C_i \ \ , \ \cup_1^N X_i\subseteq X \ .  $$
More generally, let $K\subset\R^N$ be a closed set. The set of subpartitions $\underline{\cal P}^\psi_K$ is defined by
\be\label{daa}
\underline{\cal P}^\psi_K:= \cup_{\vm\in K} \underline{\cal P}^{ \vec{\psi}}_{\vm} \ .  \ee
\par
For each agent $i$ and $x\in X$  we associate the {\it profit} $\phi_i(x)$ of $x$ for this agent. Again
$\phi_i:X\rightarrow \R$
are measurable functions.
The profit of agent $i$ under a given partition is
$$F_i(X_i):=\int_{X_i}\phi_i(x)d\mu  \ . $$
The total profit of all agents is $${\cal F}_N(\overrightarrow{X}):=\sum_1^N F_i(X_i)  \ . $$
The  object of the game is to maximize the total profit, that is,
\be\label{kol}\max_{\overrightarrow{X}} {\cal F}_N(\overrightarrow{X})\ee
 over all admissible partitions subjected in $\underline{\cal P}^\psi_K$.

 The  paradigm for the selection  problem is as follows:
 \begin{enumerate}
 \item
 Maximize the function
\be\label{mainrashi}\vm \mapsto\uXi^*(\vm)\ee
where $\uXi^*$ given by (\ref{subgmk}),  on $\uvmS_I\cap K$.
   \item
  For a maximizer $\vm$ of (\ref{mainrashi}),  evaluate the optimal subpartitions  $\oX$ realizing the maximum (\ref{kol}) \em{within}  $\uvP$.
 \end{enumerate}

  \subsection{Description of main results}

 Obviously,  if all prices $\psi_i$ are identical (say $\psi_i\equiv 1$)  then  the set $\vmS_I$ is  just the simplex $S_I$ (\ref{simI}).
In that case (\ref{gmk}) is reduced into the semi discrete Monge problem  (\ref{md}).
\par
Since the semi-discrete Monge problem is a special case of the Monge problem, a lot is known on its solvability and uniqueness.
The  essential  condition for solvability and uniqueness of the classical Monge problem is   the {\it twist condition}
which, in the present case
(and for a smooth $\phi_i$ on a smooth manifold $X$)
takes the form
\be\label{twist1}\phi_i-\phi_j \ \ \ \text{has no critical point} \ \forall i\not= j \ .  \ee
 see  [\ref{G}, \ref{Mc}, \ref{MT}, \ref{Ri} \ ....].
 \par
 The  twist condition for non-smooth  $\phi_i$  an abstract topological measure space $X$  takes the form
 \be\label{1res} \mu\left(x\in X \ ;  \phi_i(x)-\phi_j(x)=r\right)=0\ee
 for any $i\not= j\in I$ and for any $r\in\R$ (Section \ref{backM}, Theorem \ref{classM}).

 The generalization for this in the case of individual price takes the form
  \be\label{2res} \mu\left(x\in X \ ;  \phi_i(x)-\phi_j(x)=\alpha\psi_i(x)-\beta\psi_j(x)\right)=0\ee
  for any $i\not= j\in I$ and any $\alpha, \beta\in\R$ (Theorem \ref{main3}, Section \ref{scunique}). Indeed, (\ref{2res}) is reduced to (\ref{1res}) where
  $\vpsi$ is a constant.
  \par
In the case of  subpartitions  we need an additional assumption to guarantee the unique solvability, namely
  \be\label{3res} \mu\left(x\in X \ ;  \phi_i(x)=\alpha\psi_i(x)\right)=0\ee
  for any $\alpha\in\R$ and any $i\in I$.  (Theorem \ref{main3}-(ii), Section \ref{scunique}).  In particular, we need the condition
   \be\label{4res} \mu\left(x\in X \ ;  \phi_i(x)=r\right)=0\ee
  for any $r\in\R$ and any $i\in I$, in addition to (\ref{1res}) to obtain the unique solvability of the subpartition version  of the Monge problem. (Corollary \ref{corunique}, Section \ref{backM}).

 In contrast, (\ref{2res}, \ref{3res}) are not enough, in general, for the unique solvability  in the general case.   The additional condition
   \be\label{5res} \mu\left(x\in X \ ;  \alpha\psi_i(x)-\beta\psi_j(x)=0 \right)=0\ee
   for any $\alpha,\beta\in\R$, $\alpha^2+\beta^2>0$ and $i\not=j\in I$, together with (\ref{2res}, \ref{3res}),  are enough to guarantee the unique solvability of the problems introduced above (sec. \ref{sip}- \ref{sselection}).

   \subsection{Structure of the paper}
   In Section \ref{weak} we relax the  notion  of (sub)partition to that of a  weak    (sub)partition. In  Theorem \ref{weak=strong},  section \ref{ps},  we prove that the weak (sub)partition  and strong (sub)partition sets are the same. In section \ref{dual} we characterize these sets  using a dual formalism.
   \par
  Section \ref{optimweak} deals with optimal weak (sub)partitions.  In section \ref{optimweak1} we set up the condition for the existence of
  optimal weak (sub)partitions and prove the existence of such subpartition for the selection problem (Theorem \ref{main2}). In sections \ref{sduoptim} and
  \ref{sduoptim1} we use the dual formulation to characterize the optimal weak sub(partition) (Theorem \ref{main new}).
  \par
  In Section \ref{optimstrong} we discuss strong (sub)partitions. Section \ref{optimstrong1} sets up the assumption (\ref{5res}) for the existence of unique strong
  partition for any $\vm$ in the boundary of the (sub)partition set $\vmS_I$, ($\uvmS_I$), in Proposition \ref{uniquestrong}. In section
  \ref{scunique} we prove the main result for  uniqueness of optimal strong (sub)partition- Theorem \ref{main3}, and for the optimal selection (\ref{mainrashi}) where $K$ is a convex set - Theorem \ref{thselect}. Finally, in section \ref{backM} we discuss the Monge selection problem in light of the above results and prove the uniqueness of an optimal subpartition for the Monge problem under conditions (\ref{1res}, \ref{4res}), in
  Corollary \ref{corunique}.
  \subsection{Notations and conventions}\label{not}
\begin{description}
\item{i)} Unless otherwise stated explicitly, any assumption cited below is valid form its citation point to  the rest of the text.
\item{ii)} $I:= \In$. $\R^I:= \R^N$.
\item{iii)} For  $J\subseteq I$,
$ \R^J=\left\{ \vp=(p_1, \ldots p_{N})\in \R^{N}  \ \  ; \  \  p_i=0 \ \text{if} \ \ i\not\in J\  \right\} $.
 \item{iv)} The partial order relation $\vp>\vq$ ($\vp\geq\vq$) on $\R^J$ means $p_i>q_i$ ($p_i\geq q_i$) for any $i\in J$.
 \item{v)} $\R^J_+:= \{\vp\in \R^J \ ; \vp  \geq \vec{0} \}$.
\item{vi)}
 $(X, {\cal B}, \mu)$ is a compact Polish space ${\cal B}$ is the Borel$-\sigma$ algebra and $\mu$ is a Borel non-atomic measure.
  \item{vii)}  $\vpsi:= (\psi_1\ldots \psi_N)\in C(X; \R^N)$.
\item{viii)} $\vmu:= (\mu_1, \ldots \mu_N)$ where $\mu_i$ are non-negative Borel measures on ${\cal B}(X)$.
\item{ix)}
${\cal P}^{ \vec{\psi},w}_{\vm}:= \left\{\vec{\mu} \ ; \ \  \ \ \int_\Omega \psi_i d\mu_i=m_i \ \ , \ \ \sum_1^N\mu_i=\mu \right\}$.
 \item{x)}
$\underline{\cal P}^{ \vec{\psi},w}_{\vm}:= \left\{\vec{\mu} \ ; \ \  \ \ \int_\Omega \psi_i d\mu_i=m_i \ \ , \ \ \sum_1^N\mu_i\leq \mu \right\}$
\item{xi)}  $\vmSw_I:= \left\{ \vm\in \R^I \ \ ; \ \ \ {\cal P}^{w, \vec{\psi}}_{\vm}\not=\emptyset\right\}$. \ \item{xii)}
 $\uvmSw_I:= \left\{ \vm\in \R^I \ \ ; \ \ \ \underline{\cal P}^{w, \vec{\psi}}_{\vm}\not=\emptyset\right\}$.
\end{description}

 \section{Weak (sub)partitions}\label{weak}

\subsection{Back to Kantorovich}
The Monge problem (\ref{mk}, \ref{mkT})  is relaxed into the {\it Kantorovich problem} as follows:
maximize of the linear functional
$$  \int_{X\times Y} c(x,y)d\pi(x,y): \Pi(\mu,\nu) \rightarrow \R$$
where $\Pi(\mu,\eta)$ is the convex set of measures on $X\times Y$ whose marginals are $\mu, \nu$, that is
$$ \pi(A\times Y)= \mu(A) \ \ \ ; \ \ \ \pi(X\times B)=\nu(B)$$
for all measurable sets $A\in {\cal B}_X$($B\in{\cal B}_Y$).\par
Again, in the special case where $Y$ is a discrete  space $Y=\{ y_1, \ldots y_N\}$ and $\nu(\{y_i\}):= m_i\geq 0$, the set $\Pi(\mu,\nu)$ is  reduced into the set of decompositions of the measure $\mu$  into $n$ non-negative measures
 $$ {\cal P}^{w}_{\vm}:= \left\{\vec{\mu}:=(\mu_1, \ldots \mu_N) \ ; \ \  \ \ \int_X d\mu_i=m_i \ \ , \ \ \sum_1^N\mu_i=\mu \right\} \ , $$
 Indeed, $\pi\in\Pi(\mu,\nu)$ iff $\pi=\sum_1^N \mu_i\delta_{\{y_i\}}$ where $\vec{\mu}\in  {\cal P}^w_{\vm}$.

  Note that the set of partitions ${\cal P}_{\vm}$ can be embedded in ${\cal P}^{w}_{\vm}$ by identifying a set $X_i\in {\cal B}$ with the measure $\mu$  restricted to $X_i$, that is, $\mu_i:=\mu\lfloor X_i$, whence $\int_{X_i} d\mu=\int_Xd\mu_i$.

In the same way we consider the set of  {\it relaxed} (weak)  partitions corresponding to $\vec{\psi}$.  Here (\ref{part}) is generalized into
 \be\label{partweak}{\cal P}^{ \vec{\psi},w}_{\vm}:= \left\{\vec{\mu}:=(\mu_1, \ldots \mu_N) \ ; \ \  \ \ \int_\Omega \psi_i d\mu_i=m_i \ \ , \ \ \sum_1^N\mu_i=\mu \right\} \ . \ee
where, again, $\vm\in\R^I$.
 Let also $\vmS_I$   (\ref{SN}) generalized into
 \be\label{SNw} \vmSw_I:= \left\{ \vm\in \R^I \ \ ; \ \ \ {\cal P}^{w, \vec{\psi}}_{\vm}\not=\emptyset\right\} \ . \ee

Naturally, (\ref{gmk}) is generalized into
\be\label{wgmk} \Xi^*_{\phi,w}(\vm):=  \sup_{\vec{\mu}}\left\{ \sum_1^N\int_{X} \phi_i(x)d\mu_i \ \ ; \ \ \ \vec{\mu}\in {\cal P}_{\vm}^{w, \vec{\psi}}\right\}\ee
and  $ \Xi^*_{\phi,_w}(\vm)=-\infty$ iff $\vec{m}\not\in \vmSw_I$.

In analogy to (\ref{subpart}-\ref{subSN}) we also define the {\it weak subpartition}
 \be\label{subpartweak}\underline{\cal P}^{ \vec{\psi},w}_{\vm}:= \left\{\vec{\mu}:=(\mu_1, \ldots \mu_N) \ ; \ \  \ \ \int_\Omega \psi_i d\mu_i=m_i \ \ , \ \ \sum_1^N\mu_i\leq \mu \right\} \  \ee
 and
 \be\label{subSNw} \uvmSw_I:= \left\{ \vm\in \R^I \ \ ; \ \ \ \underline{\cal P}^{ \vec{\psi},w}_{\vm}\not=\emptyset\right\} \ , \ee
 $${\Xi_{\phi,w}^+}(\vm):=  \sup_{\vec{\mu}}\left\{ \sum_1^N\int_{X} \phi_i(x)d\mu_i \ \ ; \ \ \ \vec{\mu}\in \underline{\cal P}_{\vm}^{w, \vec{\psi}}\right\}$$
$ {\Xi_{\phi,w}^+}(\vm)=-\infty$ iff $\vec{m}\not\in \underline{S}_N^{\vec{\psi},w}$.\par
Since, as remarked above,  any (sub)partition  $\overrightarrow{X}\in {\cal P}^{ \vec{\psi}}_{\vm}$ ($\overrightarrow{X}\in \underline{\cal P}^{ \vec{\psi}}_{\vm}$) induces a weak (sub)partition $\vec{\mu}\in {\cal P}^{  \vec{\psi},w}_{\vm}$ ($\vec{\mu}\in \underline{\cal P}^{ \vec{\psi},w}_{\vm}$)
via
$\mu_i:= \mu\lfloor X_i$ it follows
\be\label{oneincl}  \vmS_I\subseteq  \vmSw_I \ \ , \ \ \  \uvmS_{I}\subseteq  \underline{S}^{ \vec{\psi},w}_I \ . \ee


\subsection{Properties of the partition set}\label{ps}

 \begin{lemma}\label{noempty}
The sets $\vmS_I \ , \ \  \uvmSw_I\subset \R^I$ are  compact and convex.
\end{lemma}
\begin{proof}
Since  $|m_i|=|\int_X\psi_id\mu_i|\leq \|\psi_i\|_\infty \int_X d\mu =  \|\psi_i\|_\infty$, so $\vmSw_I$ is bounded. Compactness follows  from the weak-$C^*$compactness of the set of probability measures on a compact set. Convexity follows directly from the definition.
\end{proof}
Recalling the definition of the strong (sub)partition sets (\ref{SN}, \ref{subSN}) we now prove
\begin{theorem}\label{weak=strong}
$$\vmS_I
  = \vmSw_I \  \  \  \text{and} \  \uvmS_I
  =  \uvmSw_I$$
\end{theorem}
\begin{proof} We  have to prove the opposite inclusion of (\ref{oneincl}) .   If $\vm\in \vmSw_I$, consider the set of weak partitions $\ {\cal P}^{ \vec{\psi},w}_{\vm}$.   By Radom-Nikodym Theorem,   any $\vec{\mu}= (\mu_1, \ldots \mu_N)\in
  \ {\cal P}^{ \vec{\psi},w}_{\vm}$  is characterized by   $\vec{h}= (h_1, \ldots h_N)$ where $h_i$ are $\mu$-measurable functions, $\mu_i=h_i\mu$,
 satisfying
  $0\leq h_i\leq 1$ on $X$. Moreover we have $\sum_1^N h_i=1$ $\mu$-a.e on $X$. Now $\ {\cal P}^{ \vec{\psi},w}_{\vm}$ is convex and compact  in the weak  topology.
  By Krein-Milman Theorem there exists an exposed point of   $\ {\cal P}^{ \vec{\psi},w}_{\vm}$ . We  show that for an exposed point, $h_i\in\{0,1\}$ $\mu$-a.e on $X$, for all $i\in I$.

  Assume a set $D\subset X$ on which both $h_1>\eps$ and $h_i>\eps$ for some $i\not= 1$.  Since $h_1+h_i\in[0,1]$ it follows also that $h_1$, $h_i$ are smaller than $1-\eps$ on $D$ as well. Using Lyapunov partition theorem [\ref{ly}]  we can find a subset $C\subset D$ such that
 $ \int_C\psi_1d\mu_1=\int_D \psi_1d\mu_1/2$ and$ \int_C\psi_id\mu_i=\int_D \psi_id\mu_i/2$. Set
 $w:= {\bf 1}_{D}-2{\bf 1}_{C}$ where ${\bf 1}_A$ stands for the indicator function of a measurable set $A\subset X$. It follows that  $w$ is supported on $D$ , $\|w\|_{\infty,D} =1$
 and $\int_X w\psi_1d\mu_1=\int_Xw\psi_id\mu_i=0$. By assumption, $h_1(x)\pm \eps w(x) \in [0,1]$ and $h_i(x)\pm \eps w(x) \in [0,1]$ for any $x\in D$.
  Set  \\
  $\vec{\mu}_1:= \left( \mu_1+\eps  w\mu, \mu_2, \ldots, \mu_i-\eps w\mu, \ldots \mu_N\right)$ and  \\ $\vec{\mu}_2:= \left( \mu_1- \eps w\mu, \mu_2, \ldots, \mu_i\
   +\eps w\mu, \ldots \mu_N\right)$.
  Then both $\vec{\mu}_1, \vec{\mu}_2$ are in  $\ {\cal P}^{ \vec{\psi},w}_{\vm}$ and $\vec{\mu}= \frac{1}{2}\vec{\mu}_1 + \frac{1}{2}\vec{\mu}_2$. This is in contradiction to the assumption that $\mu$ is an exposed point. It follows that  either  $h_i=0$ or $h_1=0$ $\mu$-a.e. Since $i$ is arbitrary and $\sum_I h_j=1$ $\mu$-a.e. it follows that $h_1\in \{0,1\}$
      $\mu$-a.e, hence $h_j\in\{0,1\}$ for any $j\in I$   $\mu$-a.e.     \  The proof    $ \underline{S}_I^{\vpsi}
  =  \underline{S} _I^{\vpsi,w}$ follows identically.   \end{proof}
\subsection{Dual representation of weak (sub)partitions}\label{dual}

Let now, for $\vp= (p_1, \ldots p_N)\in\R^I$
\be\label{xi0} \xi_0(x, \vp):= \max_{i\in I} p_i\psi_i(x): X\times\R^I\rightarrow \R\ee
\be \label{11} \xi_0^+(x,\vp):= \max(\xi_0(x,\vp), 0)\ee
\be\label{0Xi0} \Xi_0(\vp):= \int_X \xi_0(x, \vp) d\mu:\R^I\rightarrow\R \ \ . \ee
\be\label{0Xi0+} \Xi_0^+(\vp):= \int_X \xi_0^+(x, \vp) d\mu(x) \ . \ee
\begin{theorem}\label{main1}
 $\vm\in \vmS_I$ (res.  $\vm\in \uvmS_I$) if and only if
\be\label{ineq0} \ a) \ \  \Xi_0(\vp)-\vec{m}\cdot\vp\geq 0 \ \ \ ; \ \  b) \ \ res. \ \Xi^+_0(\vp)-\vec{m}\cdot\vp\geq 0 \ee
for any $\vp\in\R^I$. Here $\vec{m}\cdot\vp:= \sum_1^N p_im_i$
\end{theorem}

\begin{cor}\label{main2cor}
$\vm$ is an inner point of $\vmS_I$ (res. $\uvmS_I$) iff $\vp=0$ is a strict minimizer of (\ref{ineq0}-a) (res (\ref{ineq0}-b)).
\end{cor}
\noindent
{\it Proof of Corollary~\ref{main2cor}}: Since $\Xi_0(\vp) -\vm\cdot \vp$ is an homogeneous function and $\Xi_0(\vec{0})=0$, it follows that $\vec{0}$ is a minimizer of (\ref{ineq0}) for any $\vm\in \vmS_I$. If it is a strict minimizer then $\Xi_0(\vp) -\vm\cdot \vp>0$ for any $\vp\not=0$ hence there exists a neighborhood of $\vm$ for which $\Xi_0(\vp) -\vm^{'}\cdot \vp>0$ for any $\vm^{'}$ in this neighborhood, so $\vm^{'}\in \vmS_I$ by Theorem~\ref{main1}. Otherwise, there exists $\vp_0\not=0$ for which $\Xi_0(\vp_0) -\vm\cdot \vp_0=0$. Then $\Xi_0(\vp_0) -\vm^{'}\cdot \vp_0<0$ for any $\vm^{'}$ for which $(\vm-\vm^{'})\cdot\vp_0<0$. By Theorem \ref{main1} it follows  that $\vm^{'}\not\in\vmS_I$ so $\vm$ is not an inner point of $\vmS_I$. The second case is proved similarly.
$\Box$\par
The set $\vmS_I$ may contain inner points.
 As an example, consider the case where $N=2$, $\psi_1$ and $\psi_2$ are continuous, positive functions  and there exists pair of point $x,y\in X$ such that $\psi_2(x)-\psi_1(x)= \psi_1(y)-\psi_2(y)>0$. If $x$ is in the support of $\mu_1$ and $y$ in the support of $\mu_2$ then  $\vmS_I$ contains an interior point. Indeed, we can move a neighborhood of $x$ from $1$ to $2$, and a neighborhood of $y$ from $2$ to $1$. This way we increased both $m_1$ and $m_2$ to obtain $(m_1^{'}, m_2^{'})\in \vmS_I$ satisfying $m_1^{'}>m_1$ and $m_2^{'}> m_2$. On the other hand we can evidently increase one of them (say $m_1$)  while decreasing $m_2$ by transferring a mass from $2$ to $1$. Then we  obtain $(m_1^{''}, m_2^{''})\in \vmS_I$ satisfying $m_1^{''}>m_1$ and $m_2^{''}< m_2$. By convexity, $\vmS_I$ contains the triangle whose vertices are $(m_1, m_2), (m_1^{'}, m_2^{'}), (m_1^{''}, m_2^{''})$, and in particular an interior point.
\vskip .2in
To prove Theorem~\ref{main1} we need some auxiliary lemmas:
\begin{lemma}
$\Xi_0$ and $\Xi_0^+$ are  convex functions on $\R^I$.
\end{lemma}

\begin{proof} By definition, $\xi_0$ and $\xi_0^+$  are convex function in $\vp$ for any $x\in X$. Hence  $\Xi_0$, $\Xi_0^+$  are convex as well from definition (\ref{0Xi0}, \ref{0Xi0+}).
\end{proof}
\begin{lemma}
If $\vm\in \vmS_I$ then
$$ \Xi_0(\vp)-\vec{m}\cdot\vp\geq 0$$
for any $\vp\in\R^I$. Likewise, if $\vm\in \uvmS_I$ then
$$ \Xi^+_0(\vp)-\vec{m}\cdot\vp\geq 0$$
for any $\vp\in\R^I$.
\end{lemma}
\begin{proof}
 Assume $\vm\in \vmS$.  Since $\vmS_I=\vmSw_I$ by Theorem \ref{weak=strong} then, by definition, there exists $\vmu\in\vwP$ such that
$\int_X\psi_id\mu_i=m_i$. Also, from (\ref{partweak}) ($\sum_1^N \mu_i=\mu$) and (\ref{0Xi0})
$$ \Xi_0(\vp) = \int_X \xi_0(x, \vp) d\mu = \sum_1^N\int_X\xi_0(x,\vp) d\mu_i$$
while from (\ref{11}) $\xi_0(x, \vp)\geq p_i\psi_i(x)$  so
$$ \Xi_0(\vp) \geq  \sum_1^Np_i\int_X\psi_i d\mu_i= \vp\cdot \vm \ . $$
The case for $\Xi_0^+$ is proved similarly.
\end{proof}
In order to prove the second direction of Theorem~\ref{main1} we need the following definition of {\it regularized maximizer}:
\begin{defi}
Let $\vec{a}\in \R^I$. Then, for $\eps>0$,
$$ max_\eps(\vec{a}):= \eps\ln\left(\sum_{i\in I} e^{a_i/\eps}\right)$$
\end{defi}
\begin{lemma}\label{maxreg}
For any $\eps>0$ $max_\eps(\cdot)$ is a smooth convex function on $\R^I$. In addition  $\max_{\eps_1}(\vec{a})\geq\max_{\eps_2}(\vec{a})\geq  \max_{i\in I}( a_i)$ for  any $\vec{a}\in\R^I$, $\eps_1>\eps_2>0$ and
\be\label{limmaxeps}\lim_{\eps\searrow 0}max_\eps(\vec{a})= \max_{i\in I}  a_i \ . \ee
\end{lemma}
\begin{proof}
Follows from
\be\label{maxeps} max_\eps(\vec{a})= \max_{\vec{\beta}}\left\{ -\eps\sum_1^N \beta_i\ln\beta_i + \vec{\beta}\cdot\vec{a}\right\}\ee
where the maximum is taken on the simplex $0\leq \vec{\beta}$, $\vec{\beta}\cdot\vec{1}= 1$. Note that the maximizer is
$$ \beta^0_i=\frac{e^{a_i/\eps}}{\sum_je^{a_j/\eps}} < 1$$
for $i\in I$.  Since
$\sum_1^N\beta_i\ln\beta_i \leq 0$,  the term in brackets in (\ref{maxeps}) is monotone non-decreasing in $\eps>0$. Finally, (\ref{limmaxeps}) follows from the Jensen's inequality via $-\sum_1^N\beta_i\ln\beta_i \leq \ln N$.
\end{proof}
\begin{defi}\label{def32}
\be\label{eps11} \xi_\eps(x, \vp):= max_{\eps}\left( p_1\psi_1(x), \ldots p_N\psi_N(x)\right): X\times\R^I\rightarrow \R\ee
\be\label{eps22} \Xi_\eps(\vp):= \int_X \xi_\eps(x, \vp) d\mu: \R^I\rightarrow\R \ \ . \ee
Also, for each $\vp\in\R^I$ and $i\in I$ set
\be\label{defmuieps} \mu_i^{(\vp)}(dx):= \frac{e^{p_i\psi_i(x)/\eps}}{\sum_{j\in I} e^{p_j\psi_j(x)/\eps}} \mu(dx)\ee

Likewise
\be\label{eps11+} \xi^+_\eps(x, \vp):= max_{\eps}\left( p_1\psi_1(x), \ldots p_N\psi_N(x), 0\right): X\times\R^I\rightarrow \R\ee
\be\label{eps22+} \Xi^+_\eps(\vp):= \int_X \xi^+_\eps(x, \vp) d\mu: \R^I\rightarrow\R \ \ . \ee
and
\be\label{defmuieps+} \mu_i^{(\vp,+)}(dx):= \frac{e^{p_i\psi_i(x)/\eps}}{1+\sum_{j\in 1} e^{p_j\psi_j(x)/\eps}} \mu(dx)\ee
\end{defi}
Since $\max_\eps$ is smooth and convex due to lemma~\ref{maxreg} it follows from the above definition via an explicit differentiation.
\begin{lemma}\label{regXi}
For each $\eps>0$, $\Xi_\eps$ (res. $\Xi^+_\eps$) is a  convex and $C^\infty$ on $\R^I$. In addition
$$ \frac{\partial \Xi_\eps(\vp)}{\partial p_i}= \int_X\psi_i(x) d\mu_i^{(\vp)}\ \ \ res. \ \ \frac{\partial \Xi^+_\eps(\vp)}{\partial p_i}= \int_X\psi_i(x) d\mu_i^{(\vp,+)} $$
\end{lemma}
The proof of Theorem~\ref{main1} follows from the following Lemma
\begin{lemma}\label{correg}
For any $\eps, \delta>0$ and $\vm\in\R^I$
\be\label{Xicor}\vp\rightarrow \Xi_\eps(\vp) + \frac{\delta}{2}|\vp|^2-\vm\cdot\vp\ee
is a  strictly convex function on $\R^I$. In addition
\be\label{XiepsgeqXi0} \Xi_\eps(\vp) + \frac{\delta}{2}|\vp|^2-\vm\cdot\vp\geq \Xi_0(\vp)-\vm\cdot\vp+ \frac{\delta}{2}|\vp|^2\ee
so, if  (\ref{ineq0}) is satisfied,  then $\vp\rightarrow \Xi_\eps(\vp) + \frac{\delta}{2}|\vp|^2-\vm\cdot\vp$
is a coercive function as well. The same statement holds for $\Xi_\eps^+$ as well.
\end{lemma}
\par\noindent {\it Proof of Theorem~\ref{main1}} \\
From Lemma~\ref{correg} we obtain at once the existence of a minimizer $\vp_{\eps,\delta}\in\R^I$ of (\ref{Xicor}) for any $\eps,\delta>0$, provided (\ref{ineq0}) holds.  Moreover, from Lemma~\ref{regXi} we also get for that minimizer $\vp^{\eps,\delta}$ satisfies
\be\label{mi=delXi}m_i= \frac{\partial\Xi_\eps}{\partial p_i^{\eps,\delta}}+\delta p^{\eps,\delta}_i = \int_X \psi_id \mu_i^{(p^{\eps,\delta})}+\delta p^{\eps,\delta}_i\ee
By convexity of $\Xi_\eps$:
$$ \nabla\Xi_\eps(\vp) \cdot \vp \geq \Xi_\eps(\vp) - \Xi_\eps(\vec{0})$$
Multiply  (\ref{mi=delXi}) by $\vp^{\eps,\delta}$ to obtain
 \be\label{ineq1}\vp^{\eps,\delta}\cdot \nabla\Xi(\vp^{\eps,\delta})+ \delta \left| \vp^{\eps,\delta}\right|^2-\vm\cdot \vp^{\eps,\delta}=0\geq \Xi_\eps(\vp^{\eps,\delta}) - \Xi_\eps(\vec{0}) + \delta \left| \vp^{\eps,\delta}\right|^2-\vm\cdot \vp^{\eps,\delta}\ee
It follows from (\ref{ineq0}, \ref{XiepsgeqXi0},\ref{ineq1}) that
$$  - \Xi_\eps(\vec{0}) + \delta \left| \vp^{\eps,\delta}\right|^2\leq 0$$
hence
$$ \delta\left| \vp^{\eps,\delta}\right| \leq \sqrt{\delta}\sqrt{\Xi_\eps(\vec{0})} \ . $$
Hence (\ref{mi=delXi}) implies
$$ \lim_{\delta\rightarrow 0} \int_X \psi_id \mu_i^{(p^{\eps,\delta})} = m_i $$
By compactness of $C^*(X)$ and since $\sum_1^N\mu_i^{(p^{\eps,\delta})}=\mu$ via (\ref{defmuieps} ) we can choose a subsequence $\delta\rightarrow 0$ along which the limits
$$ \lim_{\delta\rightarrow 0}  \mu_i^{(p^{\eps,\delta})} := \mu_i^\eps$$
holds. It follows that
$$ \sum_1^N\mu_i^{(p^{\eps})}=\mu \ \ \ ; \ \ \ \int_X \psi_i d\mu_i^\eps=m_i \ . $$
Again, the proof for $\vm\in \uvmS_I$ is analogous.
$\Box$
\section{Weak optimal  (sub)partitions}\label{optimweak}
\subsection{Existence and characterization of weak (sub)partitions}\label{optimweak1}
Let $K\subset \R^I$ be a closed set. Recall
\be\label{daa2}
\underline{\cal P}^\psi_K:= \cup_{\vm\in K} \underline{\cal P}^{ \vec{\psi}}_{\vm} \ .  \ee

\begin{assumption}\label{assd}
The components of the function  $\vphi=(\phi_1, \ldots \phi_N) :X\rightarrow \R^I$
are upper sami continuous (usc)  and bounded on $X$.
\end{assumption}
\begin{theorem}\label{main2}
There exists a weak subpartition $\vmu$ which maximize
 the total profit
$ \int_X \vphi\cdot d\vmu$ in $\underline{\cal P}^\psi_K$.
\end{theorem}
The proof of Theorem \ref{main2} is almost immediate. Since $\psi_i$ are continuous  by standing assumption, the set $\underline{\cal P}^\psi_K$ is weakly closed.   Since $\phi_i$ are u.s.c by Assumption \ref{assd},  the limit of a maximizing sequence is a maximizer.
\subsection{Dual representation}\label{sduoptim}
Our next object is to characterize the set of optimal (sub)partitions. For this we turn back to the dual formulation.
\par\noindent
   Define the function
$\xiphi:X\times \R^I\rightarrow \R$ as
\be\label{xi+} \xiphi(x, \vp) := \max \left\{ \phi_1(x) + p_1\psi_1(x), \ldots , \phi_N(x) + p_N\psi_N (x)\right\} \ . \ee
Likewise
\be\label{xi+0} \xiphi^+(x, \vp) := \max \left\{ \phi_1(x) + p_1\psi_1(x), \ldots , \phi_N(x) + p_N\psi_N (x) , 0\right\}\ee
 Set
\be\label{Xiphi} \Xi_\phi (\vp):= \int_X \xiphi(x, \vp) d\mu(x):\R^I \rightarrow \R\ee
\be\label{Xiphi+} \Xi_\phi^+ (\vp):= \int_X \xiphi^+(x, \vp) d\mu(x):\R^I \rightarrow \R\ee
and $\Xi_\phi^*$, $ {(\Xi_\phi^+}^*):\R^I\rightarrow \R\cup \{-\infty\}$ as
\be\label{xistar} \Xi_\phi^*(\vm) = \inf_{\vp\in \R^I} \left[ \Xi_\phi(\vp) - \vm\cdot\vp\right] \ \ \ ; \ \ \
 {\Xi_\phi^+}^*(\vm) = \inf_{\vp\in \R^I} \left[ \Xi_\phi^+(\vp) - \vm\cdot\vp\right]
\ee
for $\vm\in\R^I$.

Recall that the {\it essential domain} of the concave function $F:\R^I\rightarrow \R\cup \{-\infty\}$ is the set  $\{\vm; \ F(\vm)>-\infty\}$.
\begin{lemma}\label{essd}
$\Xi_\phi^*$ (res. $\uXi^*$) is a concave function on $\R^I$.  The essential domain of \ $\Xi_\phi^*$ (res.  $\uXi^*$ )
 \  is $\vmS_I$ (res. $\uvmS_{I}$).
\end{lemma}
\begin{proof}
 Comparing the definitions of $\Xi_\phi$ to that of $\Xi_0$ we obtain \\ $\Xi_\phi(x, \vp) - \|\vphi\|_\infty\leq \Xi_0(x,\vp)\leq \Xi_\phi(x, \vp) +\|\vphi\|_\infty$ for any $x\in X$ and any $\vp\in \R^I$. It follows
 $ \Xi_0(\vp) -\|\vphi\|_\infty\leq \Xi_\phi(\vp)\leq \Xi_0(\vp) + \|\vphi\|_\infty$ for any $\vp\in \R^I$ as well.  It follows that $\Xi_\phi(\vp)-\vp\cdot\vm$ is bounded from below iff $\Xi_0(\vp)-\vp\cdot\vm$ is bounded from below.  Note that  $\Xi_0(\vp)-\vp\cdot\vm$ is bounded from below on $\R^I$ iff  $\Xi_0(\vp)-\vp\cdot\vm\geq 0 $ on $\R^I$. Theorem~\ref{main1} then implies that $\vm\in \vmS_I$ iff $\vm$ is in the essential domain of $\Xi^*_\phi$. Same proof for $\uXi^*$.
\end{proof}
\subsection{From duality to  optimal partition}\label{sduoptim1}
We now investigate the sub-gradient of $\Xi_\phi$ and $\uXi$. Recall that $\vm\in\partial_{\vp}F$ iff
 $$F(\vq)-F(\vp) \geq \vm\cdot(\vq-\vp) \ $$
 for any $\vq\in\R^I$. \par
 Let us consider the {\it positive simplex of measures}
$$ \underline{\cal P}:= \left\{ \vec{\mu}= (\mu_1, \ldots \mu_N) , \ \ \mu_i\geq 0, \ \ \ \sum_1^N \mu_i\leq\mu\right\}$$
$$ {\cal P}:= \left\{ \vec{\mu}= (\mu_1, \ldots \mu_N) , \ \ \mu_i\geq 0, \ \ \ \sum_1^N \mu_i=\mu\right\}$$
For each $\vec{\mu}\in\underline{\cal P}$ we consider the vector
\be\label{vmeasure}\vm(\vec{\mu}):= \left( \int \psi_1d\mu_1, \ldots \int \psi_Nd\mu_N\right)\in \R^I \ . \ee
\begin{lemma}\label{writeXi}For any $\vp\in\R^I$ there exists ${\cal P}_{\vp}\subset {\cal P}$, ${\cal P}_{\vp}\not=\emptyset$,
(res. $\underline{\cal P}_{\vp}\subset \underline{\cal P}$, $\underline{\cal P}_{\vp}\not=\emptyset$,) such that
\begin{description}
\item{i)} $\vm\in \partial_{\vp}\Xi_\phi$ (res.  $\vm\in \partial_{\vp}\uXi$) iff $\vm = \vm(\vec{\mu})$ for some $\vec{\mu}\in {\cal P}_{\vp}$
(res. $\vec{\mu}\in \underline{\cal P}_{\vp}$).
\item{ii)} For any $\vec{\mu}\in {\cal P}_{\vp}$ (res. $\vec{\mu}\in \underline{\cal P}_{\vp}$), $\Xi_\phi(\vp)= \vm(\vec{\mu})\cdot \vp +  \int_X \vphi\cdot d\vmu$ (res. $\uXi(\vp)= \vm(\vec{\mu})\cdot \vp +  \int_X \vphi\cdot d\vmu$).
\end{description}
\end{lemma}
\begin{proof} We present the proof for $\Xi_\phi$. The proof for $\uXi$ is analogous.
\begin{description}
\item{i)}
Let
\be\label{eps11phi} {\xiphi}^\eps(x, \vp):= max_{\eps}\left( p_1\psi_1(x)+\phi_1(x), \ldots p_N\psi_N(x)+\phi_N(x)\right): X\times\R^I\rightarrow \R\ee
\be\label{eps22phi} \Xi_\phi^\eps(\vp):= \int_X {\xiphi}^\eps(x, \vp) d\mu: \R^I\rightarrow\R \ \ . \ee
and
\be\label{defmuiepsphi} \mu_{\eps,i}^{\vphi,\vp}(dx):= \frac{\exp\left(\frac{p_i\psi_i(x)+\phi_i(x)}{\eps}\right)}{
\sum_{j=1}^N \exp\left(\frac{p_j\psi_j(x)+\phi_j(x))}{\eps}\right)} \mu(dx)\ \ , \ \ i\in \In \ . \ee
As in Lemma~\ref{regXi} we obtain that $\Xi^\eps_\phi$ is a smooth, convex function and the sequence $\Xi^\eps_\phi$ satisfies
$\lim_{\eps\rightarrow 0} \Xi^\eps_\phi=\Xi_\phi$ pointwise. In addition, Lemma~\ref{maxreg} also implies that this sequence is monotone decreasing. This implies, in particular, that $\Xi^\eps_\phi\rightarrow \Xi_\phi$ in the {\it Mosco- sense} (c.f. [\ref{A}]).
In addition
\be\label{nablagrad} \frac{\partial \Xi^\eps_\phi(\vp)}{\partial p_i}= \int_X\psi_i(x) d\mu_{\eps,i}^{\vphi, \vp}\ . \ee
By Theorem~3.66 in [\ref{A}] it follows that $\partial\Xi^\eps_\phi\rightarrow \partial \Xi_\phi$ in the sense of $G-$convergence, that is:
$$ \forall  (\vp, \vz)\in \partial \Xi_\phi , \ \ \exists (\vp_\eps, \zeta_\eps)\in\partial\Xi^\eps_\phi \ \ , \ \ \ \vp_\eps\rightarrow \vp  \ \text{and} \ \ \vec{\zeta}_\eps\rightarrow\vec{\zeta} \  \ \text{for} \ \eps\searrow 0 . $$
Since   $\partial_{\vp}\Xi^\eps_\phi=\{\nabla_{\vp}\Xi^\eps_\phi\}$  we obtain that $\vz\in\partial_{\vp}\Xi_\phi$ \  iff there exists a sequence $\vp_\eps\rightarrow\vp$ and $\nabla_{\vp_\eps}\Xi^\eps_\phi\rightarrow\vz$ as $\eps\rightarrow 0$. By (\ref{nablagrad}) $$\nabla_{\vp_\eps}\Xi^\eps_\phi=\vm\left(\overrightarrow{\mu_{\eps}^{\vphi,\vp_\eps}}\right) $$
where $\overrightarrow{\mu_{\eps}^{\vphi,\vp_\eps}}= \left( \mu_{\eps,1}^{\vphi,\vp_\eps}, \ldots \mu_{\eps,N}^{\vphi,\vp_\eps}\right)$. Let ${\cal P}_{\vp}$ be the sets of limits (in $C^*(X)$) of all sequences $$\overrightarrow{\mu_{\eps}^{\vphi,\vp_\eps}},  \ \ \ \eps\rightarrow 0 \ \ .  $$ Since $X$ is compact, ${\cal P}_{\vp}$ is non-empty for any $\vp\in\R^I$. In addition we obtain $\vm\in\partial_{\vp}\Xi_\phi$ iff there exists $\vec{\mu}\in{\cal P}_{\vp}$ for which $\vm=\vm(\vec{\mu})$.
\item{ii)} By Lemma~\ref{maxreg} with $a_i:= p_{\eps,i}\vpsi_i+\phi_i$ we obtain, after integration of $\max_\eps(\vec{a})$ over $X$ with respect to $\mu$:
$$ \Xi^\eps_\phi(\vp_\eps)= -\eps\sum_{i\in  I} \int_X \ln\left( \frac{d\mu_{\eps, i}^{\vphi, \vp_\eps}}{d\mu}\right) d\mu(x) +  \sum_{i\in I}\int_X\left( p_{\eps,i}\psi_i+\phi_i\right) d\mu_{\eps, i}^{\vphi, \vp_\eps}(x) \ . $$
Note that $\frac{d\mu_{\eps, i}^{\vphi, \vp_\eps}}{d\mu}\leq 1$ from (\ref{defmuiepsphi}).
Taking the limit $\eps\rightarrow 0$, $\vp_\eps\rightarrow \vp$ we get
\be\label{Xirun} \Xi^\eps_\phi(\vp_\eps)\rightarrow  \Xi_\phi(\vp)=  \sum_{i\in I}p_i\int\psi_i d\mu_i+\int_X\vphi\cdot d\vmu\ee
where
$$\vec{\mu}\in\{\lim_{\eps\rightarrow 0} \overrightarrow{\mu_{\eps}^{\vphi, \vp_\eps}}\}\in {\cal P}_{\vp}$$
 and $\lim_{\eps\rightarrow 0}\vm\left(\overrightarrow{\mu^{\vphi, \vp_\eps}}\right)= \vm(\vec{\mu})\in \partial_{\vp}\Xi_\phi$. The limit  (\ref{Xirun}) then takes the form
$$\Xi_\phi(\vp)= \vm(\vec{\mu})\cdot\vp+ \int_X\vphi\cdot d\vmu\ \ . $$
Again, the alternative case holds similarly.
\end{description}
\end{proof}
Let $\widehat{\cal P}$ (res. $\widehat{\underline{\cal P}}$) be the weak $(C^*)$ closure of the union of all ${\cal P}_{\vp}$
(res.  $\underline{\cal P}_{\vp}$)
for $\vp\in\R^I$:
\be\label{defhatP} \widehat{\cal P}:= \overline{\cup_{\vp\in\R^I}{{\cal P}_{\vp}}}^{C^*} \ \ \ , \ \ res.  \  \underline{\widehat{\cal P}}:= \overline{\cup_{\vp\in\R^I}{\underline{\cal P}_{\vp}}}^{C^*} \ \  . \ee
\begin{lemma}\label{writeXiw}
For any $\vm\in \vmS _I$ (res. $\vm\in \uvmS _I$) there exists $\vec{\mu}\in  \widehat{\cal P}$ (res. $\vec{\mu}\in  \underline{\widehat{\cal P}}$) for which $\vm=\vm(\vec{\mu})$. In particular, this $\vmu$ is a maximizer  of $\int_X\vphi\cdot d\vmu$ in ${\cal P}^{\psi,w}_m$ (res. $\underline{\cal P}^{\psi,w}_m$) and satisfies
$$\int_X\vphi\cdot d\vmu= \Xi_\phi^*(\vm) \ \ \ \ \ res. \ \ \int_X\vphi\cdot d\vmu= \uXi^*(\vm) \ . $$
\end{lemma}
\begin{proof}
Following the argument of Lemma~\ref{correg}, set
\be\label{setone}\vp\rightarrow \Xi_\phi(\vp) + \frac{\delta}{2}|\vp|^2-\vm\cdot\vp\ee
for some $\delta>0$. By Lemma~\ref{essd}, $\vm\in \vmS_I $ iff
$$ \Xi_\phi(\vp) + \frac{\delta}{2}|\vp|^2-\vm\cdot\vp\geq\Xi_\phi^*(\vp)+ \frac{\delta}{2}|\vp|^2$$
so  $\vp\rightarrow \Xi_\phi(\vp) + \frac{\delta}{2}|\vp|^2-\vm\cdot\vp$
is a convex coercive function. Hence there exists $\vp_\delta\in\R^I$ which minimize (\ref{setone}),
\be\label{setxitt}
 \Xi_\phi(\vp_\delta) + \frac{\delta}{2}|\vp_\delta|^2-\vm\cdot\vp_\delta= \min_{\vp\in\R^I} \left[
 \Xi_\phi(\vp) + \frac{\delta}{2}|\vp|^2-\vm\cdot\vp\right] \ee
 and
\be\label{mi2=delXi} \vm\in\partial_{\vp_\delta}\Xi_\phi + \delta\vp_\delta \ . \ee
By Lemma~\ref{writeXi}-(i) it follows that there exists $\vec{\mu}_\delta\in {\cal P}_{\vp_\delta}$ for which
$\vm=\vm(\vec{\mu}_\delta) + \delta\vp_\delta $.
We now proceed as in the proof of Theorem~\ref{main1}.
By  the definition of $\partial_p\Xi_\phi$:
$$ \partial_{\vp}\Xi_\phi(\vp) \cdot \vp \geq \Xi_\phi(\vp) - \Xi_\phi(\vec{0})$$
Multiply  (\ref{mi2=delXi}) by $\vp_\delta$ to obtain
 \be\label{ineq7} \partial_{\vp_\delta}\Xi_\phi\cdot\vp_\delta+ \delta \left| \vp_\delta\right|^2-\vm\cdot \vp_\delta=0\geq \Xi_\phi(\vp_\delta) - \Xi_\phi(\vec{0}) + \delta \left| \vp_\delta\right|^2-\vm\cdot \vp_\delta\ee
It follows from (\ref{setxitt},\ref{ineq7}) that
$ \delta \left| \vp^{\eps,\delta}\right|^2$ is bounded uniformly in $\delta>0$, so
$ \delta\left| \vp_\delta\right| \leq C\sqrt{\delta} \ $
for some $C>0$ independent of $\delta$.
Hence (\ref{mi2=delXi}) implies $\partial_{\vp_\delta}\Xi_\phi\rightarrow \vm$ as $\delta\rightarrow 0$. Hence $\vm(\vec{\mu}_\delta)\rightarrow\vm$.
By compactness of $C^*(X)$  we can choose a subsequence $\delta\rightarrow 0$ along which $\vec{\mu}_\delta$ converges to some $\vec{\mu}\in \widehat{\cal P}$ for which $\vm=\vm(\vec{\mu})$.
\end{proof}
\begin{theorem}\label{main new} There exists a maximizer of  $\int_X\vphi\cdot d\vmu$ in  $\vwP$ and any such maximizer is in  $ \widehat{\cal P}$. Likewise,  there exists a maximizer of  $\int_X\vphi\cdot d\vmu$ in  $\uvP$ and any such maximizer is in $ \underline{\widehat{\cal P}}$.
\end{theorem}
\begin{proof}
First, any $\mu\in\vwP$ satisfies
$$\int_X\vphi\cdot d\vmu \leq \Xi_\phi(\vp)-\vp\cdot\vm $$
for any $\vp\in\R^I$. Indeed,  since
$$ \phi_i(x)\leq \xi(x,p) - p_i\psi_i(x) \ \ i\in I$$
 we get
\begin{multline}\label{stat}\int_X\vphi\cdot d\vmu \leq \sum_1^N \int_X \xi(x, \vp) d\mu_i - \sum_1^N p_i\int_X\psi_i d\mu_i \\ \leq \int_X \xi(x, \vp)( \sum_1^N d\mu_i)- \vp\cdot \vm\leq \Xi_\phi(\vp)-\vp\cdot\vm\leq \Xi_\phi^*(\vm) \ .  \end{multline}
\par
Let now $\vm\in \vmS_I$ By Lemma \ref{writeXiw} there exists $\vmu\in  \widehat{\cal P}$ such that $\vm(\vmu)=\vm$.
 By definition (\ref{defhatP}) there exists  a sequence $\vp_n\in\R^I$ such that $\vec{\mu}=\lim_{n\rightarrow\infty}\vec{\mu}^n$ where $\vec{\mu}^n\in {\cal P}_{\vp_n}$.

In particular, $\vm_n:= \vm(\vec{\mu}^n)\rightarrow \vm$. By lemma~\ref{writeXi}-(ii) we obtain that
 $\Xi_\phi(\vp_n)= \vm(\vec{\mu}^n)\cdot \vp_n +  \int_X \vphi\cdot d\vmu^n$. From Lemma~\ref{writeXi}-(i)
 $$\Xi_\phi^*(\vm_n)= \Xi_\phi(\vp_n)- \vm_n\cdot \vp_n =  \int_X \vphi d\vmu^n \  $$

  Taking the limit $n\rightarrow\infty$ and the lower-semi-continuity of $\Xi_\phi^*$  we get
  $$\Xi_\phi^*(\vm)\leq  \int_X \vphi d\vmu \   \ . $$
  This, with (\ref{stat}), implies that $\vec{\mu}$ is the maximizer.
\end{proof}

\section{Strong  (sub)partitions}\label{optimstrong}
\subsection{Structure of the strong partition sets}\label{optimstrong1}
\begin{assumption}\label{mainass3} \
For any $i\in I$ and $x\in X$ $\psi_i>0$ and is positive.  In addition, for any $i\not= j\in I$
$$\mu\left[x\in X; \alpha\psi_i(x)+\beta\psi_j(x)=0\right]=0$$
 for   any $\alpha, \beta\in\R$, $\alpha^2+\beta^2> 0$,
\end{assumption}
\begin{lemma}\label{51}
Under Assumption \ref{mainass3},  $\Xi_0$ is differentiable at any point $\vp:= (p_1, \ldots p_N)$ for which $\Pi_1^N p_i\not= 0$. In particular
\be\label{diffm} \frac{\partial\Xi_0}{\partial p_i}(\vp)= \int_{X_i(\vp)} \psi_i d\mu \  \ee
is continuous,  where $X_i(\vp):= \{ x\in X \ ; \ p_i\psi_i(x)= \xi_0(x,\vp)\}$ is a strong partition.
\par
If, in addition, $\vp>\vo$ then $\Xi_0^+$ is differentiable at $\vp$ as well and
\be\label{udiffm} \frac{\partial\Xi^+_0}{\partial p_i}(\vp)= \int_{X_i(\vp)} \psi_i d\mu \ .  \ee
 \end{lemma}
\begin{proof} (\ref{udiffm}) follows from (\ref{diffm}) by definition (compare (\ref{xi0}) to (\ref{11}), using the assumption $\psi_i>0$).
Assumption \ref{mainass3}
 yields the existence of a strong partition $\overrightarrow{X}(\vp):= \left( X_1(\vp), \ldots X_N(\vp)\right)$ in $ {\cal P}_{\vm}^{\vpsi}$ associated with each $\vp$:
 \be\label{Xiphi} X_i(\vp):= \{ x\in X \ ; \ p_i\psi_i(x)= \xi_0(x,\vp)\} \ . \ee
where $\xi_0$ as defined in (\ref{xi0}). In particular $\mu(X_i(\vp)\cap X_j(\vp))=0$ for $i\not= j$.
 Note that $\xi_0(x,\vp)$ is differentiable a.e   and
$$\frac{\partial\xi_0(x,\vp)}{\partial p_i} =\left\{
  \begin{array}{cc}
    \psi_i(x) & \text{if} \ x\in X_i(\vp) \ \text{a.e} \\
    0 &  \text{if} \ x\not\in X_i(\vp) \ \text{a.e}  \\
  \end{array}
\right.
$$
A direct integration of the above over $X$ yields (\ref{diffm}).
Under Assumption \ref{mainass3}, the sets $X_i(\vp)$ are  continuous with $\vp$ in the Hausdorff metric at $\vp\not= 0$,  hence
it yields that  the right side of (\ref{diffm}) is, indeed, continuous, hence $\Xi_0$ is differentiable at any $\vp$ satisfying the assumption of the Lemma. The same proof holds for $\Xi_0^+$ where this time
$$\uX_i(\vp):= \{ x\in X \ ; \ p_i\psi_i(x)= \xi^+_0(x,\vp)\} \ $$
 is a partition.
 \end{proof}

\begin{proposition}\label{uniquestrong}
Assume $\vm>\vo$.
Under assumption \ref{mainass3},  if  $\vm\in\partial \vmS_I$  then
\begin{description}
\item{i)} $\vm$ is an exposed point in $\vmS_I$. That is, $\vm$ is not an interior point of any segment contained in $\vmS_I$.
\item{ii)} There exists a unique partition in $\vwP$. Moreover, this partition is a strong one.
 \end{description}
 \end{proposition}
 \begin{lemma}\label{u=}
 Under assumption \ref{mainass3},  if  $\vm\in\partial \uvmS_I$, $\vm > \vo$,   then  $\vm\in\partial \vmS_I$.  Moreover, $\underline{\cal P}^{\psi,w}_m= {\cal P}^\psi_m$.
 \end{lemma}
 \begin{cor}\label{v=}    
If $\vm$ is supported on $J\subset I$ so $\vm_J>\vo$ (see section \ref{not}-(iii)  and either $\vm\in\partial  \vmS_J$ or $\vm\in\partial  \uvmS_J$ then the conclusion of Proposition \ref{uniquestrong} hold.
\end{cor}
 \begin{proof} {\it of Proposition \ref{uniquestrong}}: \
 Using Corollary \ref {main2cor} we obtain that if $\vm\in\partial \vmS_I$ there exists $\vp^0\not= \vo$  for which $\Xi_0(\vp^0)-\vm\cdot\vp^0=0\leq \Xi_0(\vp)-\vm\cdot\vp $
 for any $\vp\in \R^I$.  We claim that if $\vm>\vo$ then $\vp^0$  can be chosen to satisfy  the assumption of Lemma \ref{51}.  In particular, we prove that either $\vp^0>\vo$ or $\vp^0<\vo$.
 \par
 Assume that, say, $p^0_1>0$. Since $\psi_1>c$  on $X$ for some $c>0$ by assumption, then $\xi_0({\vp}^0,x)\geq p^0_1\psi_1(x)>p_1c$  on $X$. If $p^0_j\leq 0$ for some $j\not=1$, let $\eps>0$ for which $(p_j^0+\eps)\psi_j< p^0_1\psi_1$ on $X$.
 Then by definition  $\xi_0(\vp^0,x)=\xi^0(\vp^0+\eps\vec{e_j}, x)$ on $X$.
 Here $\vec{e}_j$ is the unit  coordinate vector pointing in the positive $j$ direction. hence $\Xi_0(\vp^0)=\Xi_0(\vp^0+\eps\vec{e}_j)$
  so
  $$\Xi_0(\vp^0+\eps \vec{e}_j) -\vm\cdot(\vp^0+\eps\vec{e}_j)= \Xi_0(\vp^0) -\vm\cdot\vp_0 - \eps m_j  =-\eps m_j \ . $$
  Since $m_j>0$ by assumption it follows that we get a contradiction to   $\vp^0\in\vmS_I$ by Theorem \ref{main1}.
    \par
  Alternatively, if $p^0_1<0$ and $p_j^0\geq 0$ for some $j\not=1$, then
   $\Xi_0(\vp^0+\eps\vec{e}_1)=\Xi_0(\vp^0)$  for any $0<\eps <-p^0_1$ so

   $$\Xi_0(\vp^0+\eps \vec{e}_1) -\vm\cdot(\vp^0+\eps\vec{e}_1)= \Xi_0(\vp^0) -\vm\cdot\vp_0 - \eps m_1 <0 $$
   as well. Hence either $\vp^0>\vo$ or $\vp^0<\vo$ and, in particular, the condition of Lemma \ref{51} is satisfied.
   \par\noindent {\it Proof of (i)}: \\
  Suppose now that $\partial\vmS_I$ contains an interval centered  at $\vm>\vo$. In particular there exists $\vm_1, \vm_2\in\partial\vmS_I$, $\vm_1\not=\vm_2$ such that $\vm=(\vm_1+\vm_2)/2$.
 Let $\vp^0$ corresponding to $\vm$ as above:
  \be\label{p3m1} \Xi(\vp^0)-\frac{\vm_1+\vm_2}{2}\cdot\vp^0=0 \ . \ee
 Since $\vm_1, \vm_2\in \vmS_I$ we get by  Theorem \ref{main1}
  \be\label{p3m12} \Xi(\vp^0)-\vm_1\cdot\vp^0\geq 0 \ \ \ ; \ \ \ \Xi(\vp^0)-\vm_2\cdot\vp^0\geq 0 \ .  \ee
Averaging  these two inequalities we get
$$ \Xi(\vp^0)-\frac{\vm_1+\vm_2}{2}\cdot\vp^0\geq 0 \ $$
and, from (\ref{p3m1}) we get that the two inequalities in (\ref{p3m12}) are, in fact, equalities:
$$ \Xi(\vp^0)-\vm_1\cdot\vp^0= 0 \ \ \ ; \ \ \ \Xi(\vp^0)-\vm_2\cdot\vp^0= 0 \   $$
which implies that $\vm_1, \vm_2\in\partial_{\vp^0}\Xi_0$. In particular $\Xi_0$ is not differentiable at $\vp^0$, which is a contradiction to Lemma~\ref{51}. Hence $\vm_1=\vm_2$.
\par\noindent{\it Proof of (ii)}:

   From  Lemma \ref{51} we also get that
   $$ X^0_i(\vp^0):= \{ x\in X \ ;  p^0_i\psi_i(x)= \xi_0(x,\vp^0)\}  $$
   is a strong partition. If $\vm\in\partial\uvmS$ and  $\vm > \vo$ then, necessarily, $\vp^0>\vo$.

   We now show that  any weak partition in $\vwP$  is the strong partition given by $\oX^0$.  Indeed, if $\vec{\mu}\in \vwP$ , then
 $$ \Xi_0(\vp^0)= \int_X\xi_0(x,  \vp^0)d\mu(x)= \sum_1^N \int_X\xi_0(x, \vp^0)d\mu_i (x) \geq \sum_1^N p^0_i\int_X\psi_id\mu_i = \vp^0\cdot\vm \ . $$
 Since
 $\Xi_0(\vp^0)=\vp^0\cdot\vm$, it follows that
 $$\sum_1^N  \int_X \left( \xi_0( x, \vp^0) -p^0_i\psi_i(x)\right)d\mu_i(x)= 0 \ . $$

 Note that  $\xi_0(\vp,x)\geq p_i\psi_i(x)$ for any $i\in I$ and a.e $x\in X$ with strong inequality  only  for $x\in X^0_j(\vp^0)$, $j\not= i$,  by definition of $\xi_0$. Hence $\mu_i=h_i\mu$ where $h=0$ on $X-X^0_i(\vp^0)$ $\mu$-a.e.  Since
$\sum_1^N h_i=1$ $\mu$-a.e, it follows that, necessarily, $h_i$ is the indicator function of $X^0_i(\vp^0)$. In particular, $\vec{\mu}$ is a strong partition, and is a singleton in $ {\cal P}^{w,\vpsi}_{\vm}$.
 \end{proof}
 \begin{proof}   {\it of Lemma \ref{u=}}: \\
 Following the proof of Proposition \ref{uniquestrong} we get  the existence of $\vp^0>0$ for which $\Xi_0^+(\vp^0)-\vm\cdot\vp^0=0$. If $\vmu\in \underline{\cal P}^{\psi,w}_m$ is a weak subpartition, then as in the above proof we get
  $$ \Xi^+_0(\vp^0)= \int_X\xi^+_0(x,  \vp^0)d\mu(x)\geq  \sum_1^N \int_X\xi^+_0(x, \vp^0)d\mu_i (x) \geq \sum_1^N p^0_i\int_X\psi_id\mu_i = \vp^0\cdot\vm \ . $$
  In particular
  $$ \int_X\xi^+_0(x,  \vp^0)d\mu(x)=  \sum_1^N \int_X\xi^+_0(x, \vp^0)d\mu_i (x) \ . $$
  Since $\xi_0^+(x,\vp^0)$ is positive and continuous on $X$ and $\sum_1^N \mu_i\leq \mu$ it follows that $\vmu$ is, in fact, a weak partition.
 \end{proof}
 \subsection{Uniqueness of optimal strong (sub)partitions}\label{scunique}
 \begin{assumption}\label{mainass2} \  $\phi_i\in C(X)$ for all $ i\in I $.
\begin{description}
\item{i)} For any $i,j\in I$ and any $\alpha, \beta \in\R$,   \\ $\mu\left(x\in X \ ; \ \  \alpha\psi_i(x)-\beta\psi_j(x)+ \phi_i(x)-\phi_j(x)=0\right)=0$ .
    \item{ii)}  For any $i\in I$ and any $\alpha \in\R$,   \\ $\mu\left(x\in X \ ; \ \  \phi_i(x)=\alpha\psi_i(x)\right)=0$ .

             \end{description}
\end{assumption}

Recall (\ref{xi+},\ref{xi+0}).
 For each $\vp\in\R^I$ let
\be\label{Xiphiphi} X_i(\vp):= \{ x\in X \ ; \ p_i\psi_i(x)+ \phi_i(x) = \xiphi (x,\vp)\} \ , \ee
\be\label{uXiphiphi} \uX_i(\vp):= \{ x\in X \ ; \ p_i\psi_i(x)+ \phi_i(x) = \xiphi^+ (x,\vp)\} \ , \ee

 By Assumption \ref{mainass2}-(i) it follows that $\overrightarrow{X}(\vp)$ is, indeed, a strong partition for {\em any}  $\vp\in\R^I$.
  Likewise, Assumption  \ref{mainass2}-(i,ii) implies that $\overrightarrow{\uX}(\vp)$ is  a strong subpartition.
  In particular, $\mu(X_i(\vp)\cap X_j(\vp))=\mu(\uX_i(\vp)\cap \uX_j(\vp))=0$ for $i\not= j$. Moreover, (\ref{diffm}, \ref{udiffm}) are generalized into
 \be\label{diffmphi}  \frac{\partial\Xi_\phi}{\partial p_i}(\vp)= \int_{X_i(\vp)} \psi_i d\mu \ \ , \ \ \frac{\partial\uXi}{\partial p_i}(\vp)= \int_{\uX_i(\vp)} \psi_i d\mu \  \ee
where the right sides of (\ref{diffmphi}) are continuous in $\vp$. It follows
\begin{lemma}\label{difcor}
Under Assumption~\ref{mainass2}(i), $\Xi_\phi$ is differentiable on $\R^I$. If, in addition,  Assumption \ref{mainass2}(ii) is granted, then $\uXi$ is differentiable as well.
\end{lemma}

\begin{theorem}\label{main3} \ .
\begin{description}
\item{i)}
Let $\vm$ be an interior point of $\vmS_I$. Under Assumption~\ref{mainass2}(i) , there exists a unique partition in ${\cal P}^{\psi,w}_m$ which maximize
$\int_X\vphi\cdot d\vmu$, and this partition is a strong one.
\item{ii)} If  $\vm$ be an interior point of $\uvmS_I$
and,
 in addition, Assumption \ref{mainass2}(ii) is granted, then for
 there exists a unique subpartition in $\underline{\cal P}^{\psi,w}_m$ which maximize
$\int_X\vphi\cdot d\vmu$, and this subpartition is a strong one.
\item{iii)}
If, in addition, Assumption  \ref{mainass3} is granted that both (i, ii) hold for any $\vm\in \vmS_I$ ($\vm\in\uvmS_I$).
\end{description}
\end{theorem}
\begin{proof}
We may assume that $\vm>\vo$ for otherwise, if $m_i>0$ for $i\in J\subset I$ and $m_i=0$ for $i\not\in J$, we can restrict our discussion from $I$ to $J$ (c.f Corollary \ref{v=}).
\par
 Same argument holds if $\vm\in \partial\uvmS_J$. Hence we assume that $\vm$ is an interior  point of $\vmS_J$ (res. $\uvmS_J$).
\par
{\it (i)}: \
Recall that, for any $\vmu\in {\cal P}^{\psi,w}_m$,
$$ \int_X\vphi\cdot d\vmu \leq \Xi_\phi^*(\vm):= \inf_{\vp\in\R^J}\left[ \Xi_\phi(\vp)-\vm\cdot\vp\right]$$
and $\vmS_J$ is the essential domain of  $\Xi_\phi^*$. By Lemma \ref{writeXiw} any maximizer satisfies the equality above $\Xi_\phi^*(\vm)=\int_X\vphi\cdot d\vmu$.

If $\vm$ is an interior point then (see [\ref{BC}]) there exists $\vp\in\R^J$ for which the equality
 $$\Xi_\phi(\vp)-\vm\cdot\vp= \Xi_\phi^*(\vm)$$
 holds. For any $\vmu$ (in particular, for the maximizer) we get from the definition of $\Xi_\phi$
 \be\label{eq==}\Xi_\phi(\vp)-\vm\cdot\vp = \sum_1^N\int_X \left( \xi_\phi(x, \vp) - p_i\psi_i(x)\right)d\mu_i\ee
so, by Lemma \ref{writeXiw} any maximizer satisfies
\be\label{so} \sum_1^N\int_X\left( \xi_\phi(x,\vp)-p_{i}\psi_i(x) -\phi_i(x)\right)d\mu_i(x)=0 \ . \ee
 By Assumption \ref{mainass2}-(i) and (\ref{Xiphiphi}), the $i$ integrand above is positive on $X-X_i(\vp)$ and a.e zero on $X_i(\vp)$, so $\mu_i$ is supported on $X_i(\vp)$. Since
 \be\label{eq++++}\int_{X_i(\vp)}\psi_id\mu= \frac{\partial \Xi_\phi}{\partial p_i}(\vp) = m_i= \int_X\psi_id\mu_i\ee
 it follows that $\mu_i=\mu\lfloor X_i(\vp)$, that is, $\vmu$ is a strong partition. \\
 \par\noindent {\it (ii)}
  \ In the case $\vmu\in \underline{\cal P}^{\psi,w}_m$ (\ref{eq==}) turns into an inequality
$$\uXi(\vp)-\vm\cdot\vp \geq  \sum_1^N\int_X \left( \xi^+_\phi(x, \vp) - p_i\psi_i(x)\right)d\mu_i$$
but $\xi_\phi^+-p_i\psi_i-\phi_i\geq 0$ on $X$ by definition (\ref{xi+0}) so
$$\uXi(\vp)-\vm\cdot\vp \geq  \sum_1^N\int_X \left( \xi^+_\phi(x, \vp) - p_i\psi_i(x)-\phi_i(x)\right)d\mu_i+ \int_X\vphi\cdot d\vmu \geq  \uXi^*(\vm)$$
By Lemma \ref{writeXiw} again we have $\uXi^+(\vm) = \uXi(\vp)-\vm\cdot\vp$ so we have equality in (\ref{eq++++}) and the rest of the proof as above.
\par\noindent
 {\it (iii)}: If $\vm\in\partial\vmS_J$ ($\vm\in\partial\uvmS_I$) then by Proposition \ref{uniquestrong}-(ii) the set ${\cal P}^{\psi,w}_m$
 ($\underline{\cal P}^{\psi,w}_m$) is composed a unique strong (sub)partition, so the Theorem follows trivially.
\end{proof}
We turn now to the case of optimal selection.
\begin{theorem}\label{thselect}
Given a closed convex set $K\subset\R^I$. There exists a unique subpartition which optimize (\ref{kol}), and this subpartition is a strong one.
\end{theorem}
\begin{proof}
By Theorem~\ref{main3} we only have to prove the uniqueness of the maximizer of $\uXi^*$   on $\uvmS_I\cap K$ (\ref{mainrashi}).
\par
To show the
  uniqueness of this  maximizer  we use   Corollary 18.12(ii) on page 268 of [\ref{BC}]. It implies that a function $\uXi^*$  is strictly convex  in the interior of its domain $\uvmS_I$ if it is  the convex dual of a differentiable convex function. In our case $-\uXi^*$ is the convex dual of $\uXi$ which is differentiable by Corollary \ref{difcor}. Hence, if the a maximizer $\vm$ of $\uXi^*$ in the convex set $K\cap \uvmS_I$ is an interior point of $\uvmS_I$, then it is unique by its strong concavity. If, on the other hand, $\vm\in\partial \uvmS_I\cap K$ is a maximizer and $\vm>\vo$ then Proposition~\ref{uniquestrong}-(i) implies that $\vm$ is an exposed point of $\uvmS_I$. This implies that, again, this maximizer is unique. If the components of $\vm$ are not all positive then we reduce the problem to the subset $J\subset I$ which support the maximizer $\vm$ and apply Corollary~\ref{v=}.
\end{proof}
\subsection{Back to Monge}\label{backM}
It is interesting to compare Assumption~\ref{mainass2}-(i) with the {\it twist condition} (\ref{twist1}).
Recall that the Monge problem corresponds to the case where all $\psi_i$ are equal, say $\psi_i\equiv 1$ for any $i\in I$. In that case Assumption~\ref{mainass2}-(i) takes the form
\be\label{twist}\mu\left(x\in X \ ; \ \   \phi_i(x)-\phi_j(x)=r\right)=0\ee
for any $i\not= j$ and $r\in\R$. This seems to be a weaker version of (\ref{twist1}).  Theorem \ref{main3}-(i) yields the uniqueness of of optimal partition for any $\vm$ in the interior of the set $\vmS_I$. Embarrassingly, $\vmS_I$ is the simplex $S_I$ (\ref{simI}), and does not contain any interior point! Part (iii) of Theorem \ref{main3} is of no help either, since Assumption \ref{mainass3} is never satisfied in that case.....
\par
On the other hand, if we add Assumption \ref{mainass2}-(ii) which, in the above case,  takes the form
\be\label{twist2}\mu\left(x\in X \ ; \ \   \phi_i(x)=r\right)=0\ee
for any $i\in I$ and any $r\in \R$, then Theorem \ref{main3}-(ii) yields
\begin{cor}\label{corunique}
Under conditions (\ref{twist}, \ref{twist2}) there is a unique, strong  subpartition for the Monge partition problem
for any $\vm\in \underline{S}^0_I:=$
$$ \left\{ \vm\in\R^I \ ; \ \ 0\leq m_i \ , \ \ \sum_{i\in I} m_i < 1\right\} \ . $$
\end{cor}
However, it turns out that condition (\ref{twist}) alone is also sufficient for the uniqueness of strong partition in the Monge case:
\begin{theorem}\label{classM}
Suppose $\psi_1=\ldots=\psi_N\equiv 1$ and the components of $\vphi$ are continuous on $X$. If (\ref{twist}) is satisfied then there is a unique optimal partition for any $\vm\in S_I$ (\ref{simI}), and this unique partition is a strong one.
\end{theorem}
\begin{proof}
In the case under consideration, (\ref{xi+}, \ref{Xiphi}) takes the form
$$ \xi_1(x, \vp) := \max \left\{ \phi_1(x) + p_1, \ldots , \phi_N(x) + p_N\right\} \ , $$
$$\Xi_1 (\vp):= \int_X \xi_1(x, \vp) d\mu(x):\R^I \rightarrow \R \ , $$
Note that $\Xi_1$ is additively invariant under shift $$\Xi_1(\vp+\alpha\vec{1})=\Xi_1(\vp)+\alpha \ \ \forall \vp\in\R^I, \ \alpha\in\R \ . $$
So, $\Xi_1(\vp)-\vm\cdot\vp$ is invariant under such shift for any $\vm\in S_I$. Thus, we may set to zero the first coordinate $p_1$ of $\vp$ and obtain for $\Xi_1^0(p_2, \ldots p_N):= \Xi_1(0, p_2, \ldots p_N)$
$$ \Xi_1^0(p_2, \ldots p_N)-\sum_2^N m_ip_i= \Xi_1(\vp)-\vm\cdot\vp \ . $$
Now, (\ref{twist}) implies that $\Xi_1^0\in C^1(\R^{N-1})$ and the range of $\nabla\Xi_1^0$ is the whole $N-1$ simplex
$\sum_2^N m_i\leq 1$. Thus for any point $\vm\in S_I$ for which $m_1>0$, we get $(m_2, \ldots m_N)$ as an interior point in the range of
$\nabla\Xi_1^0$. This yields the proof of uniqueness as in Theorem \ref{main3}-(i).
\end{proof}
\newpage
\begin{center}{\bf References}\end{center}
\begin{enumerate}
\item\label{A} H. Attouch: \  {\it Variational Convergence for Functions and Operators}, Pitman publishing limited, 1984
\item\label{BC}  H.H. Bauschke and P.L. Combettes: {\it  Convex Analysis and Monotone Operator Theory in Hilbert Spaces}, Springer, 2011.
\item\label{CL} G. Carlier, A. Lachapelle: {\it  A Planning Problem Combining Calculus of Variations and Optimal
Transport}, Appl. Math. Optim. 63, , 1-9,
79-104. 2011
\item\label{G} W. Gangbo: {\it The Monge Transfer Problem and its Applications}, Contemp. Math., 226,  1999
\item\label{ly} A. Lyapunov: {\it Sur les fonctions-vecteurs completement additives} , Bull. Acad. Sci. URSS (6) (1940), 465-478.
\item \label{MT} X.N. Ma, N. Trudinger and X.J. Wang: {\it  Regularity of potential functions of the optimal transportation problem} Arch. Rational Mech. Anal., 177, 151-183, 2005
    \item\label{Mc} R. McCann and N.  Guillen:  {\it Five lectures on optimal transportation: Geometry, regularity and applications} \\
    http://www.math.cmu.edu/cna/2010CNASummerSchoolFiles/lecturenotes/mccann10.pdf
\item\label{mon} G.Monge:\ {\it M\'{e}moire sur la th\'{e}orie des d\'{e}blais et des remblais}, In  Histoire de l\'{A}cad\'{e}mie Royale des Sciences de Paris,  666-704, 1781
\item\label{Ri}  L. Ruschendorf and L.  Uckelmann: {\it On optimal multivariate couplings. Distributions with given marginals and moment problems},  261-273, Kluwer Acad. Publ., Dordrecht, 1997
\item \label{vil} C. Villani:  {\it Topics in Optimal Transportation}, A.M.S Vol 58, 2003
\item \label{vil1} C. Villani:  {\it  Optimal Transport Old and New}, Springer 2009

\end{enumerate}\end{document}